\documentclass[a4paper, 11pt]{amsart}
\usepackage{srcltx}
\usepackage[utf8]{inputenc}
\usepackage{amsfonts, amsthm, amsmath, amssymb,bm}

\usepackage[T1]{fontenc}
\usepackage{verbatim}
\usepackage{amssymb,amscd,amsmath,amsthm}
\usepackage{mathtools}
\usepackage{indentfirst}
\usepackage{enumerate}
\usepackage[subnum]{cases}

\usepackage{fancyhdr}

\usepackage{bbm}

\theoremstyle{plain}
\newtheorem{thm}{Theorem}[section]
\newtheorem{lem}[thm]{Lemma}

\newtheorem{prop}[thm]{Proposition}

\theoremstyle{definition}

\newtheorem*{rmk}{Remark}

\usepackage{mdframed}

\newcommand{\C}{\mathbb{C}}
\newcommand{\F}{\mathbb{F}}

\newcommand{\R}{\mathbb{R}}
\newcommand{\Z}{\mathbb{Z}}

\newcommand{\rmod}{\!\!\!\!\pmod}

\newcommand{\NUM}{\operatorname{NUM}}
\newcommand{\DEN}{\operatorname{DEN}}

\newcommand{\eps}{\varepsilon}

\usepackage[square,sort,comma,numbers]{natbib}
\usepackage{graphicx}

\title[Squarefree numbers in large arithmetic progressions]{Squarefree numbers in large arithmetic progressions}

\author{Ramon M. Nunes} \thanks{This work is supported by the DFG-SNF lead agency program grant 200021L-153647.}

\address{EPFL SB MATHGEOM TAN\\ Station 8\\ CH-1015 Lausanne\\ Switzerland}
\email{ramon.moreiranunes@epfl.ch}

\usepackage{hyperref}

\begin{document}

\begin{abstract}
We show that the exponent of distribution of the sequence of squarefree numbers in arithmetic progressions of prime modulus is $\geq 2/3 + 1/57$, improving a result of Prachar from 1958. Our main tool is an upper bound for certain bilinear sums of exponential sums which resemble Kloosterman sums, going beyond what can be obtained by the Polya-Vinogradov completion method.
\end{abstract}



\maketitle

\section{Introduction and statement of results}

%
%

%
%
%
%
%
%

\subsection{Squarefree numbers in arithmetic progressions}
Let $\mu$ denote the M\"obius function, \textit{i.e.}~$\mu$ is the multiplicative function such that for every prime number $p$ and every positive integer $\alpha$, one has,
$$
\mu(p^{\alpha})=
\begin{cases}
-1,\,\text{if }\alpha=1,\\
\;\;\;0,\,\text{otherwise}.
\end{cases}
$$
We remark that $\mu^2(n)=1$ if $n$ is squarefree and $\mu^2(n)=0$ otherwise. In this paper we are concerned with the distribution of squarefree numbers in arithmetic progressions. By the above discussion, this is equivalent to studying the distribution of the $\mu^2$ function in arithmetic progressions.

In this direction, a result of Prachar \citep{prachar1958kleinste}, subsequently improved by Hooley \citep{hooley1975note} says that
\begin{equation}\label{Hoo-sqf}
\sum_{\substack{n\leq x\\n \equiv a\rmod q}}\mu^2(n) = \frac{1}{\varphi(q)}\sum_{\substack{n\leq x\\(n,q)=1}}\mu^2(n) + O\left(\frac{X^{1/2}}{q^{1/2}} + q^{1/2+\epsilon}\right).
\end{equation}
It follows from Asymptotic formula \eqref{Hoo-sqf} that the sequence of squarefree numbers $\leq X$ is well distributed in arithmetic progressions modulo $q$ whenever
\begin{equation}\label{q<X23}
q\leq X^{2/3-\epsilon},
\end{equation}
for some fixed positive $\epsilon$. Even though it is largely believed that one should be able to replace $2/3$ by $1$, this constant has resisted any improvement since Prachar \citep{prachar1958kleinste}.

In \citep{nunes2015conjectures}, we were able to show a slight improvement, meaning that we proved that one can replace \eqref{q<X23} by $q\leq X^{2/3}(\log X)^{\delta}$, where $\delta$ is some small (but fixed) constant. The technique there was based on non-trivial upper bounds for exponential sums by Bourgain and Garaev. These upper bounds show cancellation in very short sums but the upper bound is only better than the trivial by some small power of the logarithm of the length of the sum, this is the reason for the rather modest improvement in \citep{nunes2015conjectures}.

Our main result proves that one can replace $2/3$ by $13/19 = 2/3 + 1/57$ in \eqref{q<X23}. Precisely, we have
\begin{thm}\label{2/3}
Let $\epsilon>0$ and $A >0$. Then, uniformly for $X\geq 2$, integers $a$ and prime numbers $q$ coprime with $a$ satisfying
$$
q\leq X^{\frac{13}{19}-\epsilon},
$$
we have
$$
\sum_{\substack{n\leq X\\n \equiv a\rmod q}}\mu^2(n) = \frac{1}{\varphi(q)}\sum_{\substack{n\leq X\\(n,q)=1}}\mu^2(n) + O\left(\frac{X}{q(\log X)^A}\right),
$$
In other terms, the value $\Theta=\frac{13}{19}$ is an exponent of distribution for the characteristic function of the sequence of squarefree numbers $\mu^2$ restricted to prime moduli.
\end{thm}

We believe it is helpful to compare this result with \citep[Theorem 1.1]{fouvry2015exponent} on the level of distribution of the ternary divisor function. In \citep{fouvry2015exponent}, one can see that Poisson summation and a straightforward application of the Deligne  bound for two-dimensional Kloosterman sums would already give that the ternary divisor function on integers up to $X$ is well distributed in arithmetic progressions modulo $q\leq X^{1/2-\epsilon}$. Improving the constant $1/2$ requires a way to get further cancellation than what comes from the Deligne bound and this is done by means of estimates of bilinear sums of Kloosterman sums.

In our case one sees that using only the Weil bound \eqref{Weil}, one can retrieve Hooley's result \eqref{Hoo-sqf} and again the way to get further cancellation is by means of estimates for sums of exponential sums. In the present case, the estimate needed is exactly that of Theorem \ref{expsum1} below.

In the following we discuss these sums of exponential sums from a general perspective before specializing to our case the case that interest us here.

\subsection{Sums of exponential sums}

Upper bounds for exponential sums play a major role in modern analytic number theory. The classical Weil bound for one-variable exponential sums states that for any prime number $q$, and any rational function $f\in \Z(X)$ satisfying some mild conditions, we have the upper bound
\begin{equation}\label{WeilKloo}
\sideset{}{^{\ast}}\sum_{x\rmod q}e_q(f(x))\ll q^{1/2},
\end{equation}
where the implied constant depends only on the number of roots and poles of $f$. Throughout the article, $e_q(x):=e^{2i\pi x/q}$, the $\ast$ means that we only sum over the $x$ that are not poles of $f$, and finally, $\bar x$ denotes the multiplicative inverse of $x$ modulo $q$.

A much deeper result of Deligne provides similar upper bounds for sums in several variables. Many problems in analytic number theory are reduced to obtaining estimates for exponential sums that follow directly from the Weil or the Deligne bound. However, in some problems, a straightforward application of these fails to give the desired result. One way of getting by is to take advantage of some extra summation that may be offered by the problem.

This is at the heart of a recent series of papers by Fouvry, Kowalski and Michel (\citep{fouvry2014algebraic},  \citep{fouvry2015algebraic}, \citep{fouvry2015exponent}, etc.). For instance, in \citep{fouvry2014algebraic} they prove upper bounds for sums such as
\begin{equation}\label{fkm1}
\sum_{M/2<m\leq M}\sum_{N/2<n\leq N}K(mn),
\end{equation}
where $K$ is a general \textit{algebraic trace function of bounded conductor} (see \citep{fouvry2014algebraic} for a precise statement and some examples). For instance, their result applies for hyper-Kloosterman sums, i.e. for $K(t)=\operatorname{Kl}_k$, where
\begin{equation}\label{h-kloo}
\operatorname{Kl}_k(t):=
q^{-\frac{k-1}{2}}\underset{u_1\cdots u_k=t}{\sum\ldots\sum}e_q(u_1+\ldots+u_k),\text{ if }t\neq 0.
\end{equation}
We remark that the Deligne bound $|\operatorname{Kl}_k(t)|\leq k$ is already highly non-trivial and the upper bound from \citep{fouvry2014algebraic} is saying the we can get even further cancellation when averaging as in \eqref{fkm1}. We also mention that their results apply for functions such as $K_1(t)$ and $K_2(t)$ in definition \eqref{KK} below.

The upper bounds in \citep{fouvry2014algebraic} are non-trivial as soon as $MN\geq q^{3/4+\epsilon}$. In particular one can take $M=N=q^{\theta}$ with $\theta<1/2$. This is an important threshold, since in general, a much simpler method, using orthogonality of characters can give non trivial upper bounds by only taking advantage of one of the sums. This method is usually called the completion method. See \citep{fouvry2015short} for discussions on this method and for some examples where one can go beyond this threshold for one-dimensional sums.

Sometimes one even needs to consider more general bilinear sums:
\begin{equation}\label{kms}
\sum_{m}\sum_{n}\alpha_m\beta_nK(mn),
\end{equation}
where $\bm{\alpha}=(\alpha_m)_m$ and $\bm{\beta}=(\beta_n)_n$ are sequences of complex numbers supported in $[M/2,M]$ and $[N/2,N]$ respectively. Note that the sum in \eqref{fkm1} correspond to the sequences $\alpha_m=\bm{1}_{[M/2,M]}$ and $\beta_n=\bm{1}_{[N/2,N]}$, where, for $A\subset \R$, $\bm{1}_A$ denotes its characteristic function.

In \citep{blomer2014moments} and \citep{kowalski2015bilinear}, sums such as those in \eqref{kms} are studied in the case where $K(t)$ is a hyper-Kloosterman sum.
In this paper, we are led to study the following type of bilinear sums:
\begin{equation}\label{smooth-mn2}
\sum_{M/2<m\leq M}\sum_{N/2<n\leq N}K(mn^2).
\end{equation}
Here, again, our interest lies in ranges where $M,N\leq q^{\theta}$ for some $\theta<1/2$.

Notice that the sums in \eqref{smooth-mn2}, like those in \eqref{fkm1}, are \textit{smooth}, meaning that there are no annoying terms $\alpha_m$ or $\beta_n$. It is natural to think that the techniques of \citep{fouvry2014algebraic} could be adapted to our situation. Unfortunately this is not the case, at least not in a straightforward manner. The technique in \citep{fouvry2014algebraic} uses the spectral theory of modular forms and the fact that the divisor function
$$
d(t):=\sum_{mn=t}1
$$
has an interpretation in terms of Fourier coefficients of certain Eisenstein series. Due to the lack of intepretation in terms of modular forms for the function $d_{1,2}(t):=\sum_{mn^2=t}1$, we are not able to transpose the methods of \citep{fouvry2014algebraic} to our case. Instead we will follow the methods in \citep[Section 5]{blomer2014moments}, which are in turn inspired by those of \citep{fouvry1998certaines}.

We are now ready to state our main estimate on sums of exponential sums, but first we must define the $K$-functions in which we are interested. For a prime number $q$ and integers $m$ and $n$, we let
\begin{equation}\label{Smnq}
S(m,n;q):=\sideset{}{{}^{\ast}}\sum_{u\rmod q}e_q(m{\bar u}^2+nu).
\end{equation}
If $m$ is coprime with $q$, we have the Weil bound:
\begin{equation}\label{Weil}
S(m,n;q)\leq 3q^{1/2}.
\end{equation}
For a fixed prime number $q$ and integers $a$ and $b$ coprime to $q$, we introduce the normalized sums
\begin{equation}\label{KK}
K_1(t):=q^{-1/2}S(a,bt;q)\text{ and }K_2(t):=q^{-1/2}S(at,b;q),
\end{equation}
where $S(m,n;q)$ is as in \eqref{Smnq}. As far as the notation is concerned, we forget about the depedency on $a$ and $b$ and $q$.

We prove the following:
\begin{thm}\label{expsum1}
Let $q$ be a prime number. Let $M, N\geq 1$ be such that
$$
1\leq M\leq N^2,\,\,N<q,\,\,MN^2< q^2.
$$
Let ${\bm{\alpha}}=(\alpha_m)_{m\leq M}$ be a sequence of complex numbers bounded by 1, and let $\mathcal{N}\subset [1,q-1]$ be an interval of length $N$. Finally, let $a$ and $b$ be coprime with $q$ and let $K_2(t)$ be give by \eqref{KK}. Then for any $\epsilon>0$, we have

$$
\sum_{m\leq M}\sum_{n\in \mathcal{N}}\alpha_mK_2(mn^2)\ll q^{\epsilon}\|\bm{\alpha}\|_1^{1/2}\|\bm{\alpha}\|_2^{1/2}M^{1/4}N\left(\frac{M^3N^6}{q^4}\right)^{-1/16},
$$
where the implied constant depends on $\epsilon$, and where 
$$
\|\bm{\alpha}\|_1=\sum_{m}|\alpha_m|\text{ and }\|\bm{\alpha}\|_2=\left(\sum_{m}|\alpha_m|\right)^{1/2}.
$$
\end{thm}
This can be thought of as an inhomogeneous version of \citep[Inequality (5.3)]{blomer2014moments} or \citep[Theorem 1.3]{kowalski2015bilinear}, where $K(mn)$ is replaced $K(mn^2)$.

The proof of Theorem \ref{expsum1} will be intertwined with that of
\begin{thm}\label{expsum2}
Let $q$ be a prime number. Let $M, N\geq 1$ be such that

$$
1\leq M\leq N^2,\,\,N<q,\,\,MN< q^{3/2}.
$$
Let ${\bm{\alpha}}=(\alpha_m)_{m\leq M}$ be a sequence of complex numbers bounded by 1, and let $\mathcal{N}\subset [1,q-1]$ be an interval of length $N$. Finally, let $a$ and $b$ be coprime with $q$ and let $K_1(t)$ be give by \eqref{KK}. Then for any $\epsilon>0$, we have
$$
\sum_{m\leq M}\sum_{n\in \mathcal{N}}\alpha_mK_1(mn)\ll q^{\epsilon}\|\bm{\alpha}\|_1^{1/2}\|\bm{\alpha}\|_2^{1/2}M^{1/4}N\left(\frac{M^2N^5}{q^3}\right)^{-1/12},
$$
where the implied constant depends on $\epsilon$.
\end{thm}
Notice that this is exactly \citep[Theorem 1.3.]{kowalski2015bilinear} for our modified Kloosterman sum in \eqref{KK}.

To appreciate the strength of Theorems \ref{expsum1} and \ref{expsum2}, let us assume $\alpha_m=1$ for every $m\leq M$. In this case, the bound $\ll MN$ follows directly from \eqref{Weil}. The upper bound from Theorem \ref{expsum1} (respectively \ref{expsum2}) improves on this bound, for instance, when $M=N=q^{\theta}$ with $\theta>4/9$ (respectively $\theta>3/7$). The remarkable feature is that both $4/9$ and $3/7$ are smaller than $1/2$, meaning that our methods go beyond what can be obtained by the completion method.

\section*{Structure of the article}

In the next section we make some algebraic considerations that will be useful when verifying the necessary conditions to apply a result of Hooley (see Lemma \ref{Hoo} below). These results are mostly about when certain rational functions can be written as the square of another rational function. These considerations are a bit tedious but rather elementary and are mainly based in the partial fractional decomposition for rational functions.

The third section is dedicated to bounding bilinear sums. In particular, we prove Theorems \ref{expsum1} and \ref{expsum2}. Our approach is inspired by those in \citep{fouvry1998certaines} and \citep{blomer2014moments}. Indeed, the argument in \citep{blomer2014moments} adapts here almost straightforwardly. The only extra difficulty that comes up is that in our case we need to guarantee that certain rational functions are not squares, at which point we recur to the results from Section \ref{algebraic}.

Finally, Section \ref{proofof23} is dedicated to the proof of Theorem \ref{2/3}. The main ingredient here is, as we mentioned, Theorem \ref{expsum1}, but before we can use it, some preparation is necessary. The first thing we need is a bilinear structure for $\mu^2$. This is given by the classical formula \eqref{mu-decomp}. It turns out that the term $\mu(n_2)$ plays no role in studying the problem in Theorem \ref{2/3}, which is what allows for an application of Poisson summation in both variables. Finally, we conclude by applying Theorem \ref{expsum1}.

\section{Algebraic considerations}\label{algebraic}

Let $q$ be an odd prime number and let $\F_q$ be a finite field with $q$ elements that we identify with $\Z/q\Z$ whenever is convenient. Finally, we fix $\overline{\F_q}$ an algebraic closure of $\F_q$.

The next three lemmas investigate when certain rational functions are squares. The first two are simple and follow almost directly by partial fraction decomposition. The third one is a bit more involved and will be deduced from the previous ones.

\begin{lem}\label{AB-sq}
Let $A,B,\rho_1,\rho_2\in\overline{\F_q}$ be such that $A$ and $B$ are non-zero and $\rho_1$ and $\rho_2$ are distinct. Then the rational function
$$
1+\frac{A}{(X-\rho_1)^2}+\frac{B}{(X-\rho_2)^2}
$$
is a square if and only if $A=B=(\rho_1-\rho_2)^2$.

\end{lem}

\begin{proof}
We start by noticing that

$$
\left(1+\frac{(\rho_1-\rho_2)}{(X-\rho_1)}-\frac{(\rho_1-\rho_2)}{(X-\rho_2)}\right)^2=1+\frac{(\rho_1-\rho_2)^2}{(X-\rho_1)^2}+\frac{(\rho_1-\rho_2)^2}{(X-\rho_2)}.
$$

On the other hand, suppose there exists $g(X)\in\overline{\F_q}(X)$ such that
\begin{equation}\label{AB=g2}
1+\frac{A}{(X-\rho_1)^2}+\frac{B}{(X-\rho_2)^2}=g(X)^2.
\end{equation}
We consider the partial fraction decomposition of $g(X)$. It is not difficult to see that the polynomial part of $g(X)$ must be constant and that $\rho_1$ and $\rho_2$ are the only poles of $g(X)$ and both are simple. In other words, we have
$$
g(X)=c_0 +\frac{c_1}{X-\rho_2}+\frac{c_2}{X-\rho_2},
$$
for some $c_0,\,c_1,\,c_2\in \overline{\F_q}$. Using the identity
\begin{equation}\label{easy-id}
\frac{1}{(X-\rho_1)(X-\rho_2)}=\frac{1}{\rho_1-\rho_2}\left(\frac{1}{X-\rho_1}-\frac{1}{X-\rho_2}\right),
\end{equation}
we see that
\begin{multline*}
g(X)^2=c_0^2+\frac{c_1^2}{(X-\rho_1)^2}+\frac{c_2^2}{(X-\rho_2)^2}+2c_1\left(c_0+\frac{c_2}{\rho_1-\rho_2}\right)\frac{1}{X-\rho_1}\\+2c_2\left(c_0-\frac{c_1}{\rho_1-\rho_2}\right)\frac{1}{X-\rho_2}.
\end{multline*}
Comparing it to \eqref{AB=g2}, we see that we must have
\begin{equation}\label{squaresAB}
c_0^2=1,\,c_1^2=A,\,c_2^2=B
\end{equation}
In particular $c_1,c_2\neq 0$. Furthermore,
$$
c_1\left(c_0+\frac{c_2}{(\rho_1-\rho_2)}\right)=c_2\left(c_0-\frac{c_1}{(\rho_1-\rho_2)}\right)=0,
$$
which implies that $c_1=-c_2=c_0(\rho_1-\rho_2)$. Squaring this relation and comparing it to \eqref{squaresAB} concludes the proof. 

\end{proof}

\begin{lem}\label{ABC-sq}
Let $A,B,C,\rho_1,\rho_2,\rho_3\in\overline{\F_q}$ be such that $A$, $B$ and $C$ are non-zero and $\rho_1$, $\rho_2$ and $\rho_3$ are distinct. If the rational fraction
$$
1+\frac{A}{(X-\rho_1)^2}+\frac{B}{(X-\rho_2)^2}+\frac{C}{(X-\rho_3)^2}
$$
is a square, then
$$
\frac{1}{\rho_1-\rho_2}+\frac{1}{\rho_2-\rho_3}+\frac{1}{\rho_3-\rho_1}=0.
$$
\end{lem}

\begin{proof}
Suppose there exists $g(X)\in\overline{\F_q}(X)$ such that
\begin{equation}\label{ABC=g2}
1+\frac{A}{(X-\rho_1)^2}+\frac{B}{(X-\rho_2)^2}+\frac{C}{(X-\rho_3)^2}=g(X)^2.
\end{equation}
We consider the partial fraction decomposition of $g(X)$ as before. We find out that
$$
g(X)=c_0 +\frac{c_1}{X-\rho_2}+\frac{c_2}{X-\rho_2}+\frac{c_3}{X-\rho_3},
$$
for some $c_0,\,c_1,\,c_2,\,c_3\in\overline{\F_q}$. Squaring both sides and using the identity \eqref{easy-id}, we obtain
\begin{multline}\label{squaresABC}
g(X)^2=c_0^2+\frac{c_1^2}{(X-\rho_1)^2}+\frac{c_2^2}{(X-\rho_2)^2}+\frac{c_3^2}{(X-\rho_3)^2}\\
+2c_1\left(c_0+\frac{c_2}{\rho_1-\rho_2}+\frac{c_3}{\rho_1-\rho_3}\right)\frac{1}{X-\rho_1}+2c_2\left(c_0-\frac{c_1}{\rho_1-\rho_2}+\frac{c_3}{\rho_2-\rho_3}\right)\frac{1}{X-\rho_2}\\
+2c_3\left(c_0-\frac{c_1}{\rho_1-\rho_3}-\frac{c_2}{\rho_2-\rho_3}\right)\frac{1}{X-\rho_3}.
\end{multline}
As before, we notice that $c_i\neq 0$, $i=0,1,2,3$. This implies that
$$
\begin{cases}
c_0+\frac{c_2}{\rho_1-\rho_2}+\frac{c_3}{\rho_1-\rho_3}=0,\\
c_0-\frac{c_1}{\rho_1-\rho_2}+\frac{c_3}{\rho_2-\rho_3}=0,\\
c_0-\frac{c_1}{\rho_1-\rho_3}-\frac{c_2}{\rho_2-\rho_3}=0.
\end{cases}
\\
$$
Multiplying these equations by $\frac{1}{\rho_2-\rho_3}$, $\frac{1}{\rho_3-\rho_1}$ and $\frac{1}{\rho_1-\rho_2}$ respectively and adding them up gives the result.
\end{proof}

In the proof on the next lemma, $C$ will always denote a non-zero constant that might be different at each appearance. 

\begin{lem}\label{X-sq}
Let $\alpha$ and $\beta$ be elements of $\overline{\F}_q$. Let

$$
f_{\alpha,\beta}(X):=1+\frac{1}{X^2}-\frac{1}{(\alpha X+\beta)^2}-\frac{1}{((1-\alpha)X+(1-\beta))^2}.
$$
Then there exists a set $\mathcal{E}\in {\overline{\F_q}}^2$ with $|\mathcal{E}|\leq 14$ such that for all $(\alpha,\beta)\in {\overline{\F_q}}^2\backslash \mathcal{E}$, the rational fraction $f_{\alpha,\beta}(X)$ is not a square in $\overline{\F_q}(X)$.
\end{lem}

\begin{proof}
Suppose there exists $g(X)\in\overline{\F_q}(X)$ such that
\begin{equation}\label{f=g2}
f_{\alpha,\beta}(X)=g(X)^2.
\end{equation}
\textbf{First case.} Suppose the polynomials $X$, $L(X)=\alpha X+\beta$ and $\tilde{L}(X)=(1-\alpha)X+(1-\beta)$ are non-constant and pairwise coprime.

In this case, Lemma \ref{ABC-sq} gives
\begin{equation}\label{usingABC}
\frac{\alpha}{\beta}+\frac{\alpha(1-\alpha)}{\alpha-\beta}-\frac{1-\alpha}{1-\beta}=0.
\end{equation}
We consider the partial fraction decomposition of $g(X)$. It is not difficult to see that the polynomial part of $g(X)$ must be constant and that the roots of $X$, $L(X)$ and $\tilde{L}(X)$ are the only poles of $g(X)$ and this poles are simple. In other words, we must have that
$$
g(X)=a +\frac{b}{X}+\frac{c}{L(X)}+\frac{d}{\tilde{L}(X)},
$$
for some $a,b,c,d \in \overline{\F_q}$. This and \eqref{f=g2} give
\begin{multline}\label{numerators}
X^2L(X)^2\tilde{L}(X)^2 +L(X)^2\tilde{L}(X)^2-X^2\tilde{L}(X)^2-X^2L(X)^2=\\
\big(aXL(X)\tilde{L}(X) +bL(X)\tilde{L}(X)+cX\tilde{L}(X)+dXL(X)\big)^2.
\end{multline}
In particular, $a^2=b^2=1$ and $c^2=d^2=-1$.

We remark that
\begin{equation}\label{L1XL}
L(X)-1=X-\tilde{L}(X),
\end{equation}
and since the left-hand side of \eqref{numerators} can be written as
$$
X^2\tilde{L}(X)^2(L(X)^2-1)-L(X)^2(X^2-\tilde{L}(X)^2),
$$
we see that it is divisible by $L(X)-1$. Hence the same holds for the right-hand side.

We notice that (recall \eqref{L1XL})
\begin{multline}
aXL(X)\tilde{L}(X) +bL(X)\tilde{L}(X)+cX\tilde{L}(X)+dXL(X)\equiv\\
(a+c)X^2 +(b+d)X\pmod {L(X)-1}.
\end{multline}
Therefore $L(X)-1$ divides $(a+c)X^2 +(b+d)X$. Since $L(X)-1=X-\tilde{L}(X)$, it follows that $L(X)-1$ is coprime with $X$, and thus $L(X)-1$ divides $(a+c)X +(b+d)$. Finally, since $a^2=1$ and $c^2=-1$, and $q$ is odd, we see that $a+c$ is non-zero. It follows from the above discussion that
$$
\frac{\beta-1}{\alpha}=\frac{b+d}{a+c}.
$$
By interchanging the roles of $L(X)$ and $\tilde{L}(X)$ in the above argument, leads to
$$
\frac{-\beta}{1-\alpha}=\frac{b+c}{a+d}.
$$
Since $c^2=d^2=-1$, then either $c=d$, in which case
$$
\frac{\beta-1}{\alpha}=\frac{-\beta}{1-\alpha},
$$
and hence, $\alpha+\beta=1$.

On the other hand, if $c=-d$, then
$$
\frac{\beta-1}{\alpha}\frac{-\beta}{1-\alpha}=\frac{b-c}{a+c}\frac{b+c}{a-c}=1,
$$
in which case $\alpha(1-\alpha)=\beta(1-\beta)$. That is $\alpha=\beta$ or $\alpha+\beta=1$. Notice that $\alpha=\beta$ contradicts the hypothesis that $L(X)$ and $\tilde{L}(X)$ are coprime.

Suppose we have $\alpha+\beta=1$. Then, by \eqref{usingABC}, we see that
$$
\frac{\alpha}{1-\alpha}+\frac{\alpha(1-\alpha)}{2\alpha-1}-\frac{1-\alpha}{\alpha}=0,
$$
which implies $\alpha^4-2\alpha^3+5\alpha^2-4\alpha+1=0$. We put 
$$
\mathcal{E}_1=\{(\alpha,1-\alpha)\in \overline{\F_q}^2;\; \alpha^4-2\alpha^3+5\alpha^2-4\alpha+1=0\},
$$
so that if we assume $\mathcal{E}_1\subset \mathcal{E}$, then we are done in this case.

\textbf{Second case} Suppose now that $X$, $L(X)$, $\tilde{L}(X)$ are not pairwise coprime or one of them is constant.

There are a few cases to consider, namely $\alpha\in\{0,1\}$, $\beta\in\{0,1\}$ and $\alpha=\beta$.

\begin{itemize}
\item If $\alpha=0$.
\end{itemize}
Suppose further that $\beta\neq 0,1,-1$. In this case we have
$$
f(X)=\left(1-\frac{1}{\beta^2}\right)\left(1+\frac{\beta^2}{\beta^2-1}\cdot\frac{1}{X^2}-\frac{\beta^2}{\beta^2-1}\cdot\frac{1}{(X+(1-\beta))^2}\right).
$$
But since $q$ is odd, Lemma \ref{AB-sq} implies that  $f_{\alpha,\beta}(X)$ is not a square unless $(0,\beta)\in \mathcal{E}_2$, where
$$
\mathcal{E}_2=\{(0,-1),(0,0),(0,1)\}.
$$

\begin{itemize}
\item If $\alpha=1$.
\end{itemize}
Since $f_{\alpha,\beta}=f_{1-\alpha,1-\beta}$, it follows from the previous case that $f_{\alpha,\beta}(X)$ is not a square unless $(1,\beta)\in \mathcal{E}_3$, where
$$
\mathcal{E}_3=\{(1,0),(1,1),(1,2)\}.
$$

\begin{itemize}
\item If $\beta=0$.
\end{itemize}
Suppose further that $\alpha\neq 0,1$. In this case we have

$$
f_{\alpha,\beta}(X)=1+\left(1-\frac{1}{\alpha^2}\right)\frac{1}{X^2}-\frac{1}{(1-\alpha)^2}\cdot\frac{1}{(X-\frac{1}{\alpha-1})^2}.
$$
Lemma \ref{AB-sq} now says that if $f_{\alpha,\beta}$ is a square, then

$$
-\frac{1}{\alpha^2}=-\frac{1}{(1-\alpha)^2}=\frac{1}{(1-\alpha)^2}.
$$
And since $q$ is odd, this is impossible. So that $f_{\alpha,\beta}(X)$ is not a square unless $(\alpha,0)
\in \mathcal{E}_4$, where
$$
\mathcal{E}_4=\{(0,0),(0,1)\}.
$$

\begin{itemize}
\item If $\beta=1$.
\end{itemize}
Again, by using the identity $f_{\alpha,\beta}=f_{1-\alpha,1-\beta}$, it follows from the previous case that $f_{\alpha,\beta}(X)$ is not a square unless $(\alpha,1)\in \mathcal{E}_4$, where
$$
\mathcal{E}_5=\{(1,0),(1,1)\}.
$$
\begin{itemize}
\item If $\alpha=\beta$
\end{itemize}
Suppose further that $\alpha\neq 0,1$. In this case we have
$$
f_{\alpha,\beta}(X)=1+\frac{1}{X^2}-\left(\frac{1}{\alpha^2}+\frac{1}{(1-\alpha)^2}\right)\frac{1}{(X+1)^2}.
$$
Once again by Lemma \ref{AB-sq}, we have that $f(X)$ is not a square unless
$$
1=-\left(\frac{1}{\alpha^2}+\frac{1}{(1-\alpha)^2}\right),
$$
which implies $\alpha^4-2\alpha^3+3\alpha^2-2\alpha+1=0.$ Thus it follows that $f_{\alpha,\beta}(X)$ is not a square unless $(\alpha,\alpha)\in \mathcal{E}_6$, where
$$
\mathcal{E}_6= \left\{(0,0),(1,1)\right\}\cup\{(\alpha,\alpha)\in \overline{\F_q}^2;\; \alpha^4-2\alpha^3+3\alpha^2-2\alpha+1=0\}.
$$
Then assuming $\mathcal{E}_6 \subset \mathcal{E}$ concludes this case.

We summarize by saying that putting $\mathcal{E}=\mathcal{E}_1\cup\ldots\cup\mathcal{E}_6$, so that $|\mathcal{E}|\leq 14$, we finish the proof of the lemma.
\end{proof}

We close this section with the following lemma, whose proof is to a large extent an adaptation of the argument in \citep[pages 27-29]{blomer2014moments}.

\begin{lem}\label{UV-sq}
Let $q$ be an odd prime. $\alpha$, $\beta$, $h$ be elements of $\overline{\F}_q$. Let  $F=F_{\alpha,\beta,h}$ be the rational function given by

$$
F(U,V):=\frac{1}{U^2}+\frac{1}{V^2}-\frac{1}{(\alpha U+\beta V +h)^2}-\frac{1}{((1-\alpha)U+(1-\beta)V - h)^2}.
$$
Then there exists a set $\mathcal{E}\in {\overline{\F_q}}^2$ with $|\mathcal{E}|\leq 14$ such that for all $(\alpha,\beta)\in {\overline{\F_q}}^2\backslash \mathcal{E}$ and every $h \in \overline{\F_q}$, the rational function $F(U,V)$ is \textit{well-defined} and is not \textit{composed}. That is, we cannot write
$$
F=Q\circ P,
$$
where $P(U,V)\in \overline{\mathbb{F}}_q(U,V)$ and $Q(T)\in\overline{\mathbb{F}}_q(T)$ is not a fractional linear transformation.
\end{lem}

\begin{proof}
We start by making the birational change of variables
$$
X=U/V,\;Y=V.
$$
Thus we have
$$
F(XY,Y)=\frac{1}{X^2Y^2}+\frac{1}{Y^2}-\frac{1}{(YL(X)+h)^2}-\frac{1}{(Y\tilde{L}(X)-h)^2},
$$
where we put
$$
L(X)=\alpha X+\beta;\; \tilde{L}(X)=(1-\alpha)X+(1-\beta).
$$
We need to prove that if $(\alpha,\beta)\not\in \mathcal{E}$, then $F(XY,Y)$ cannot be expressed in the form
$$
\frac{Q_1\left(P_1(X,Y)/P_2(X,Y)\right)}{Q_2\left(P_1(X,Y)/P_2(X,Y)\right)},
$$
where $P_1(X,Y),P_2(X,Y)\in\overline{\F}_q[X,Y]$ are coprime polynomials and
$$
Q_1(T)=C\prod_{\lambda}(T-\lambda)^{m(\lambda)},\,Q_2(T)=\prod_{\mu}(T-\mu)^{m(\mu)}
$$
are also coprime. Here the products are taken over the roots of $Q_1$ and $Q_2$ respectively. Moreover $m(\lambda)$ and $m(\mu)$ denote the multiplicities of these roots. Let $q_1=\deg Q_1 = \sum_{\lambda}m(\lambda)$ and $q_2=\deg Q_2 = \sum_{\mu}m(\mu)$. We remark that we can always suppose that $q_1>q_2$. If this is not the case, we simply make the change of variables
$$
T\mapsto \mu_0 + \frac{1}{T'},
$$
where $\mu_0$ is any root of $Q_2(X,Y)$.

We want to prove that $q_1=1$. We have
$$
F(XY,Y)=\frac{C\prod_{\lambda}(P_1(X,Y)-\lambda P_2(X,Y))^{m(\lambda)}}{P_2(X,Y)^{q_1-q_2}\prod_{\mu}(P_1(X,Y)-\mu P_2(X,Y))^{m(\mu)}}=:\frac{\NUM(X,Y)}{\DEN(X,Y)},
$$
with $\NUM(X,Y)$ and $\DEN(X,Y)$ coprime. In the other hand
\begin{equation}
F(XY,Y)=\frac{\NUM'(X,Y)}{\DEN'(X,Y)},
\end{equation}
where 
\begin{multline*}
\NUM'(X,Y)=
(YL(X)+h)^2(Y\tilde{L}(X)-h)^2(X^2+1)\\
-X^2Y^2((YL(X)+h)^2+(Y\tilde{L}(X)-h)^2)
\end{multline*}
and 
\begin{equation*}
\DEN'(X,Y)=X^2Y^2(YL(X)+h)^2(Y\tilde{L}(X)-h)^2	.	
\end{equation*}
In what follows we distinguish two cases.

\textbf{Case I: }$h\neq 0$.\\
Since $X$, $Y$, $YL(X)+h$ and $Y\tilde{L}(X)-h$ are relatively coprime, then $\NUM'(X,Y)$ and $\DEN'(X,Y)$ are coprime and hence equal $C\cdot\NUM(X,Y)$ and $C\cdot \DEN(X,Y)$ respectively. Comparing the expressions for $\DEN(X,Y)$ and $\DEN'(X,Y)$ we see that either $Y$ divides $P_2$, or it divides $P_1-\mu P_2$ for some $\mu$, a root of $Q_2$. In the second case, up to making a linear change of variables $T\mapsto T+\mu$, we can suppose $\mu=0$, $Y$ divides $P_1$ and $\lambda \neq 0$ for every $\lambda$ which is a root of $Q_1$. In any case we have that
$$
\NUM(X,0)=C\cdot P_1(X,0)^{q_1}\text{ or }C\cdot P_2(X,0)^{q_1}.
$$
But $\NUM(X,0)=C(X^2+1)$, thus $\NUM(X,0)$ only has simple roots. Therefore $q_1=1$. And since $q_2< q_1$, then $q_2=0$.

We proved that when $h\neq 0$, $F(U,V)$ is not composed for any $(\alpha,\beta)\in \overline{\F_q}^2$.

\textbf{Case II:} $h=0$.\\
In this case we have
\begin{equation}\label{h=0}
F(XY,Y)=\frac{L(X)^2\tilde{L}(X)^2(X^2+1)-X^2(L(X)^2+\tilde{L}(X)^2)}{X^2Y^2L(X)^2\tilde{L}(X)^2}.
\end{equation}
Suppose that $F(XY,Y)\neq 0$. Then we see that $\NUM(X,Y)$ must divide the numerator of the right-hand side of \eqref{h=0}. Hence it is independent of $Y$.

\begin{itemize}
\item Suppose $q_2>0$.
\end{itemize}
We notice that we must have that $Y$ divides $DEN(X,Y)$ and that $DEN(X,Y)$ divides $X^2Y^2L(X)^2\tilde{L}(X)^2$. Since all the factors in
$$
P_2(X,Y)^{q_1-q_2}\prod_{\mu}\left(P_1(X,Y)-\mu P_2(X,Y)\right)
$$
are coprime, we see that one of them must be divisible by $Y$ and all the others must be independent of $Y$. Now by the same argument as above, we can suppose that $\lambda=0$ is not a root of $Q_1$ and that $P_1$ and $P_2$ are two non-zero polynomials such that one of which is divisible by $Y$ and the other is independent of $Y$. But this is not possible since
\begin{equation}\label{recallNUM}
\NUM(X,Y)=C\prod_{\lambda}(P_1(X,Y)-\lambda P_2(X,Y))^{m(\lambda)},
\end{equation}
and the left-hand side is independent of $Y$ and the right-hand side cannot be.
\begin{itemize}
\item Suppose now that $q_2=0$.
\end{itemize}
This case is more delicate. We have that $Y$ divides $P_2$. The fact that $\NUM(X,Y)$ is independent of $Y$ implies that the same holds for $(P_1(X,Y)-\lambda P_2(X,Y))$ for every $\lambda$ which is a root of $Q_1$. But this implies that $Q_1$ has a unique root. Indeed, if $\lambda\neq \lambda'$, then
$$
(P_1(X,Y)-\lambda P_2(X,Y)) - (P_1(X,Y)-\lambda' P_2(X,Y))=(\lambda-\lambda')P_2(X,Y)
$$
is non-zero and divisible by $Y$ so that it is not possible for both to be independent of $Y$. Therefore, up to making the linear change of variables $T\mapsto T+\lambda$, we have that
\begin{equation}\label{F=power}
F(XY,Y)=\frac{P_1(X,Y)^{q_1}}{P_2(X,Y)^{q_1}}.
\end{equation}
We notice that since $Y\mid P_2$ and $P_2^{q_1}\mid X^2Y^2L(X)^2\tilde{L}(X)^2$, we must have $q_1=1$ or $2$. We only have to rule out the case where $q_1=2$. That is, we need to ensure that
$$
F(XY,Y)= \frac{1}{Y^2}\left(1+\frac{1}{X^2}-\frac{1}{L(X)^2}-\frac{1}{\tilde{L}(X)^2}\right)
$$
is not a square in $\overline{\F_q}(X)$. But Lemma \ref{X-sq} precisely gives a set $\mathcal{E}$ whose cardinality is bounded by of $14$ and such that if $(\alpha,\beta)\not\in \mathcal{E}$, then 
$$
\left(1+\frac{1}{X^2}-\frac{1}{L(X)^2}-\frac{1}{\tilde{L}(X)^2}\right)
$$
is not square. Thus the same holds for $F(XY,Y)$, which concludes this case.

Finally, we consider the case where $F(XY,Y)=0$. That is
$$
L(X)^2\tilde{L}(X)^2(X^2+1)-X^2(L(X)^2+\tilde{L}(X)^2)=0.
$$
By simply comparing the coefficients of degree 6 and 0, we see that this is only possible if $(\alpha,\beta)=(0,1)$ or $(1,0)$, both of which belong to the set $\mathcal{E}$ from Lemma \ref{X-sq}. This concludes the proof of Lemma \ref{UV-sq}.
\end{proof}
\section{Bounds for exponential sums}\label{exponentialsums}
In this section we prove the bounds for exponential sums on Theorems \ref{expsum1} and \ref{expsum2}. Let $q$ be an odd prime number. Let $j\geq 1$ be an integer, and let $M$ and $N$ be real numbers such that
$$
1\leq M\leq N^2,\,\,N<q,\,\,MN^j< q^{\frac{j+2}{2}}.
$$
Throughout this section we use the notation $x\sim X$ meaning the inequalities
$$
X/2<x\leq X.
$$
Let $\bm{\alpha}=(\alpha_m)$ a sequence of complex numbers supported on $m\sim M$. Let $\mathcal{N}$ be an interval of length $N$. Let further $K:\Z\rightarrow \C$ be a bounded periodic function of period $q$.

Finally, we let $S_{K,j}=\mathcal{S}_{K,j}(\bm{\alpha},M,\mathcal{N})$ be given by

$$
\mathcal{S}_{K,j}=\mathcal{S}_{K,j}(\bm{\alpha},M,\mathcal{N}):=\sum_{m\leq M}\sum_{n\in \mathcal{N}}\alpha_m K(mn^j).
$$
A simple application Cauchy's Inequality gives
\begin{equation}\label{onlyWeil}
\mathcal{S}_{K,j}\ll (\|\bm{\alpha}\|_1\|\bm{\alpha}\|_2)^{1/2}M^{1/4}N\|K\|_{\infty},
\end{equation}
where $\|K\|$ denotes the maximum of $K$ (recall that $K$ is periodic). In what follows we show how to improve upon this estimate for some specific choices of $K$ and $j$. To do so, we use Vinogradov's "shift by $ab$" technique in the following manner. Let $A,B\geq 1$ be such that
\begin{equation}\label{necessary}
AB\leq N,\,\,A^jM< q.
\end{equation}
We have
\begin{align*}
\mathcal{S}_{K,j}&=\frac{1}{AB}\sum_{a\sim A}\sum_{b\sim B}\sum_{m\leq M}\sum_{n+ab\in \mathcal{N}}\alpha_m K(m(n+ab)^j)\\
&=\frac{1}{AB}\sum_{a\sim A}\sum_{b\sim B}\sum_{m\leq M}\sum_{n+ab\in \mathcal{N}}\alpha_m K(a^jm({\bar a}n+b)^j).
\end{align*}
Suppose $I=[u,u']$ and let $g$ be an infinitely differentiable function supported on $[u-1,u'+1]$ such that $g(x)\geq 1$ for $x\in I$ and
$$
g^{j}(x)\ll x^{-j},\,\, j=0,1,2.
$$
We deduce
\begin{equation}\label{estimatesfourier}
\widehat{g}(y)\ll \min(N, |y|^{-1},|y|^{-2}).
\end{equation}
Following the lines of \citep[p.116]{fouvry1998certaines}, we see that by Fourier inversion, we have that
\begin{align*}
|\mathcal{S}_{K,j}|&\leq \frac{1}{AB}\sum_{a\sim A}\sum_{m\leq M}\sum_{n\in \mathcal{N}'}\left|\alpha_m \sum_{b\sim B}K(a^jm({\bar a}n+b)^j)g(n+ab)\right|\\
&\leq \frac{1}{AB}\sum_{a\sim A}\sum_{m\leq M}\sum_{n\in \mathcal{N}'}\frac{\left|\alpha_m \right|}{a}\int_{\R}\left|\widehat{g}\left( t/a\right)\right|\left|\sum_{b\sim B}K(a^jm({\bar a}n+b)^j)e(-bt)\right|dt.
\end{align*}
Now by \eqref{estimatesfourier} and the upper bound
$$
\int_{\R}\min(N, |y|^{-1},|y|^{-2})dy \ll \log N \leq \log q,
$$
we see  that, there exists $t\in \R$ such that
$$
\mathcal{S}_{K,j} \ll \frac{\log q}{AB}\sum_{a\sim A}\sum_{m\leq M}\sum_{n\in \mathcal{N}'}\left|\alpha_m \right| \left|\sum_{b\sim B}K(a^jm({\bar a}n+b)^j)e(-bt)\right|.
$$
We make the change of variables $r=a^jm$ and $s=\overline{a}n$. We obtain (for $\eta_b=e(-bt)$)
$$
\mathcal{S}_{K,j}\ll \frac{\log q}{AB}\sum_{r\rmod q}\sum_{s\leq A^jM}\nu(r,s)\left|\sum_{b\sim B}\eta_b K(s(r+b)^j)\right|,
$$
for
$$
\nu(r,s)=\underset{am\equiv s,\,{\bar a}n\equiv r \rmod q}{\sum_{a\sim A}\sum_{m\leq M}\sum_{n\in \mathcal{N}'}}|\alpha_m|,
$$
where $\mathcal{N}'$ is an interval containing $\mathcal{N}$ of length $2N$ and $|\eta_b|\leq 1$. Now we see that exactly as in \citep[p.116]{fouvry1998certaines} or \citep[p.26]{blomer2014moments}, we have the inequalities
$$
\sum_{r,s}\nu(r,s)\ll AN\|\bm{\alpha}\|_1\text{ and }\sum_{r,s}\nu(r,s)^2\ll q^{\epsilon}AN\|\bm{\alpha}\|_2^2.
$$
These bounds combined with another application of H\"older's inequality give
\begin{equation}\label{Holder}
AB\times \mathcal{S}_{K,j}\ll q^{\epsilon}(AN)^{3/4}(\|\bm{\alpha}\|_1\|\bm{\alpha}\|_2)^{1/2}\left(\sum_{r\rmod q}\sum_{s\leq A^jM}\left|\sum_{b\sim B}{\eta}_b K(s(r+b)^j)\right|^4\right)^{1/4}.
\end{equation}
Expanding the fourth power, we see that the double sum over $r$ and $s$ can be written as
$$
\sum_{\bm{b}\in\mathcal{B}}\eta(\bm{b})\Sigma_j(K,\bm{b}),
$$
where $\mathcal{B}$ denotes the set of quadruples $\bm{b}=(b_1,b_2,b_3,b_4)$ such that $b_i\sim B$ for $1\leq i\leq 4$,
$$
\Sigma_j(K;\bm{b}):= \sum_{r\rmod q}\sum_{s\leq A^jM}K(s(r+b_1)^j)K(s(r+b_2)^j)\overline{K(s(r+b_3)^j)K(s(r+b_4)^j)}.
$$
and the coefficients $\eta(\bm{b})$ satisfy $|\eta(\bm{b})|\leq 1$ for every $\bm{b}\in \mathcal{B}$.

We now proceed to estimate $\Sigma_j(K,\bm{b})$. In most cases we expect a lot of cancellation when we sum over $r$ and $s$ but for certain (diagonal) cases, we cannot expect this to happen (for example when $\{b_1,b_2\}=\{b_3,b_4\}$).

Let $\mathcal{B}^{\Delta}$ be a subset of $\mathcal{B}$ to be specified later and such that $\mathcal{B}^{\Delta}$ contains $\{\bm{b}\in \mathcal{B};\,(b_1,b_3)=(b_2,b_4)\}$. For those $\bm{b}\in \mathcal{B}^{\Delta}$, we do not seek for cancellation when we sum over $r$ and $s$. We simply bound everything trivially:

\begin{equation}\label{diagonal}
\sum_{\bm{b}\in\mathcal{B}^{\Delta}}\eta(\bm{b})\Sigma_j(K,\bm{b}) \leq |\mathcal{B}^{\Delta}|\times A^jMq \times\|K\|_{\infty}^4.
\end{equation}

In the non-diagonal case, i.e. $\bm{b}\in \mathcal{B}\backslash \mathcal{B}^{\Delta}$, we complete the sum over $s$ using additive characters. We thus obtain

\begin{equation}\label{completion}
\Sigma_j(K,\bm{b})\ll (\log q)\max_{0\leq h< q}\Sigma_j(K,\bm{b},h),
\end{equation}
where
\begin{equation}\label{Kbh}
\Sigma_j(K;\bm{b},h):= \underset{r,s\rmod q}{\sum\sum}\prod_{i=1}^2K(s(r+b_i)^j)\overline{K(s(r+b_{i+2})^j)}e_q(hs).
\end{equation}
In the following we will prove square-root cancellation for most of the $\bm{b}\in\mathcal{B}$.
\begin{prop}\label{completesum}
Let $q$ be an odd prime number. Let $a$ and $b$ be coprime with $q$. Let $K_1$  and $K_2$ be given by \eqref{KK}. With notation as above, there exists a choice for $\mathcal{B}^{\Delta}$ satisfying $|\mathcal{B}^{\Delta}|\ll B^2$ and for every $\bm{b}\in \mathcal{B}\backslash\mathcal{B}^{\Delta}$, and every $h\in \mathbb{F}_q$, we have the inequalities
\begin{equation}
\Sigma_1(K_1,\bm{b},h)\ll q\text{ and }\Sigma_2(K_2,\bm{b},h)\ll q,
\end{equation}
where the implied constants are absolute.
\end{prop}

\begin{rmk}
At this point it is important to notice that a simpler argument, based solely on the Weil bound for exponential sums over curves could give an upper bound $\ll q^{3/2}$ in the proposition above. However, even with optimal choices for $A$ and $B$ this would fail to give an improvement of \eqref{onlyWeil}.
\end{rmk}

\subsection{Reduction to a two-dimensional exponential sum}

From this point on, we need to specify the exact form of our $K-$function. The approach is slightly different in the two cases of Proposition \ref{completesum}. 
\begin{itemize}
    \item \textbf{Case $j=1$, $K=K_1$.}
\end{itemize}
We begin by considering the case with $j=1$. We recall that in this case we have
\begin{equation}\label{1stK}
K_1(t)=q^{-1/2}\sideset{}{{}^{\ast}}\sum_{u\rmod q}e_q\left(a{\bar u}^2+btu\right),
\end{equation}
where $a$ and $b$ are coprime with $q$. We use definition \eqref{1stK} in formula \eqref{Kbh} and perform the sum over $s$. There are two separate cases according to whether $u_1+u_2-u_3-u_4\neq 0$ or $u_1+u_2-u_3-u_4= 0$. The first part equals
\begin{equation}\label{uuuu0}
q^{-1}\underset{\substack{u_1,u_2,u_3,u_4\in \F_q\\ u_1+u_2-u_3-u_4\neq 0}}{\sum\sum\sum\sideset{}{{}^{\ast}}\sum}e_q\left( a\left({\bar u_1}^2+{\bar u_2}^2-{\bar u_3}^2-{\bar u_4}^2\right) \right)= \left(q|K_1(0)|^4 - \!\!\sum_{r\rmod q}|K_1(r)|^4\right)\ll q,
\end{equation}
by the Weil bound \eqref{Weil}. We may now focus on the second part, i.e. when $u_1+u_2-u_3-u_4=0$. We see from \eqref{uuuu0} that in the present case, Proposition \ref{completesum} is equivalent to the upper bound
\begin{equation}\label{notsimplified}
\underset{(u_1,u_2,u_3,u_4)\in W(\F_q)}{\sum\sum\sum\sideset{}{^{\ast}}\sum}e_q\Big(a\left({\bar u_1}^2+{\bar u_2}^2-{\bar u_3}^2-{\bar u_4}^2\right)\Big)\ll q,
\end{equation}
where the variety $W$ is given by the equations
$$
\begin{cases}
u_1+u_2-u_3-u_4=0\\
b_1u_1+b_2u_2-b_3u_3-b_4u_4=-\overline{b}h.
\end{cases}
$$
Assume that $\mathcal{B}^{\Delta}$ contains the set $\{\bm{b}\in \mathcal{B};\,(b_1,b_3)=(b_2,b_4)\}$. Then for every $\bm{b}\in \mathcal{B}\backslash \mathcal{B}^{\Delta}$ we either have $b_1\neq b_2$ or $b_3\neq b_4$. We assume that the second possibility holds. The other case is analogous.
Let
$$
\alpha=\frac{b_1-b_4}{b_3-b_4},\,\beta=\frac{b_2-b_4}{b_3-b_4},\,\lambda=\frac{{\bar b}h}{b_3-b_4}.
$$
Thus we can write the exponential sum on the left-hand side of \eqref{notsimplified} as
$$
\underset{u,v\rmod q}{\sum\sideset{}{^{\ast}}\sum}e_q\left(a\left({\overline{u}}^2+{\overline{v}}^2-{\overline{(\ell(u,v)-\lambda)}}^2-{\overline{(\tilde{\ell}(u,v)+\lambda)}}^2\right)\right),
$$
where
$$
\ell(u,v)=\alpha u+\beta v\text{ and }\tilde{\ell}(u,v)=(1-\alpha)u+(1-\beta)v. 
$$
By Lemma \ref{Hoo} below, \eqref{notsimplified} will follow if, for instance, we can prove that the variety

\begin{equation}\label{W(t)}
W_{\lambda}(t):=\left\{(u,v)\in \overline{\F_q}^2;\; \left({\overline{u}}^2+{\overline{v}}^2-{\overline{(\ell(u,v)-\lambda)}}^2-{\overline{(\tilde{\ell}(u,v)+\lambda)}}^2\right)=t \right\}
\end{equation}
is an irreducible curve for all but finitely many $t\in \overline{\F_q}$. We argue that for a suitable choice of the set $\mathcal{B}^{\Delta}$ this is implied by Lemma \ref{UV-sq}. Indeed, let $\mathcal{E}$ be the finite set of exceptions given by Lemma \ref{UV-sq}. If
$\mathcal{B}^{\Delta}$ contains all the solutions of the linear system

\begin{equation}\label{system}
\begin{cases}
b_1=\alpha(b_3-b_4)+b_4\\
b_2=\beta(b_3-b_4)+b_4,
\end{cases}
\end{equation}
for every $\alpha,\beta\in \mathcal{E}$, then the rational function
$$
F(U,V):=\frac{1}{U^2}+\frac{1}{V^2}-\frac{1}{(\ell(U,V) +\lambda)^2}-\frac{1}{(\tilde{\ell}(U,V) - h)^2}
$$
can not be written as $Q\circ P$, where $P$ is a rational function in two variables and $Q$ is a rational function in two variables and $P$ is a rational function in one variable which is not a fractional linear transformation. 

We argue that this implies that $W_{\lambda}(t)$ is an irreducible curve for all but finitely many $t\in\F_q$. Indeed, an argument based on L\"uroth's Theorem implies the desired result (see \citep[Proposition 2.1]{fouvry1998certaines} for details). We conclude this case by invoking the following result of Hooley (see \citep[Theorem 5]{hooley1980exponential}):
\begin{rmk}
Recall that we are considering the case where $b_3\neq b_4$. In order to take care the of the case where $b_1\neq b_2$ we must also ask that $\mathcal{B}^{\Delta}$ contains all the solutions of the dual system obtained from \eqref{system} by replacing the roles of $(b_1,b_2)$ and $(b_3,b_4)$. Note that this at most doubles the size of the set of exceptions $\mathcal{B}^{\Delta}$.
\end{rmk}
\begin{lem}\label{Hoo}
Let $q$ be a prime number. Let $f(X_1,X_2,X_3)$ and $g(X_1,X_2,X_3)$ be two rational functions over $\mathbb{F}_q$ such that
\begin{enumerate}[i)]
\item The variety $W(t)$ defined by the equation $f(X_1,X_2,X_3)=t$ and $g(X_1,X_2,X_3)=0$ is generically an absolutely irreducible curve.

\item For every specialisation of $t$ in $\overline{F_q}$, $W(t)$ is a (possibly reducible) curve.
\end{enumerate}
Then, we have the upper bound
$$
\sideset{}{{}^{\ast}}\sum_{\substack{X_1,X_2,X_3\rmod{q}\\g(X_1,X_2,X_3)\equiv 0 \rmod q}}e_q\left(f(X_1,X_2,X_3)\right)\ll q,
$$
where the implied constant depends at most on the degrees of the rational fractions $f$ and $g$.
\end{lem}
Lemmas \ref{UV-sq} and \ref{Hoo} now imply the upper bound \eqref{notsimplified} for a suitable choice of $\mathcal{B}^{\Delta}$.
\begin{itemize}
    \item \textbf{Case $j=2$, $K=K_2$.}
\end{itemize}
We now turn our attention to the case relevant to Theorem \ref{expsum1}. Let
\begin{equation}\label{2ndK}
K_2(t):=q^{-1/2}\sideset{}{{}^{\ast}}\sum_{u\rmod q}e_q\left(at\bar{u}^2+bu\right),
\end{equation}
where $a$ and $b$ are coprime to $q$.

The first thing we notice is that if $t\neq0$, then by a linear change of variables, we have that
$$
K_2(st^2)=q^{-1/2}\sideset{}{{}^{\ast}}\sum_{u\rmod q}e_q\left(ab^2s\bar{u}^2+tu\right).
$$
By using it in \eqref{Kbh} and considering the cases where $r+b_i=0$ separately, we see that
\begin{equation}\label{K2develop}
\Sigma_2(K_2,\bm{b},h)=\underset{\substack{u_1,u_2,u_3,u_4\rmod q\\\ (u_1,u_2,u_3,u_4)\in V(\F_q)}}{\sum\sum\sum\sideset{}{^{\ast}}\sum}e_q\left(b_1u_1+b_2u_2-b_3u_3-b_4u_4)\right) + O(q),
\end{equation}
where $V(\F_q)$ is the surface defined by the equations
$$
\begin{cases}
u_1+u_2-u_3-u_4=0\\ \overline{u_1}^2+\overline{u_2}^2-\overline{u_3}^2-\overline{u_4}^2=-{\overline{ab^2}}h.
\end{cases}
$$

The situation here resembles that of \citep[Theorem 1.1]{Fouvry2001general}, where very general exponential sums are considered. A direct application of their result would give a version of Proposition \ref{completesum} with the weaker bound $|\mathcal{B}^{\Delta}|\ll B^3$ for the set of exceptions.

It should still be possible to obtain Theorem \ref{expsum1} from this weaker bound but some extra work would be necessary.

We adopt a different, more elementary approach reducing to the previous case (\textit{i.e.} $K=K_1$ and $j=1$) that we discuss now.

As in the previous case, we can suppose that $b_3\neq b_4$, the case where $b_1\neq b_2$ being analogous. This allows us to write the sum on right-hand side of \eqref{K2develop} as 
$$
\underset{\substack{u_1,u_2,u_3\rmod q\\\ \overline{u_1}^2+\overline{u_2}^2-\overline{u_3}^2-\overline{(u_1+u_2-u_3)}^2\equiv-{\overline{ab^2}}h\rmod q}}{\sum\sum\sideset{}{^{\ast}}\sum}\!\!\!\!\!\!\!\!\!\!\!\!\!\!\!\!\!\!\!\!\!\!\!\!\!e_q\left((b_1-b_4)u_1+(b_2-b_4)u_2-(b_3-b_4)u_3)\right).
$$
We need to prove that $\Sigma_2(K_2,\bm{b},h)\ll q$. By arguing exactly as before, it suffices to prove that for almost every $t\in\overline{\F_q}$, the variety $W'(t)$ defined by
$$
\begin{cases}
(b_1-b_4)u_1+(b_2-b_4)u_2-(b_3-b_4)u_3=t\\ \overline{u_1}^2+\overline{u_2}^2-\overline{u_3}^2-\overline{(u_1+u_2+u_3)}^2=-{\overline{ab^2}}h.
\end{cases}
$$
is an irreducible curve. Suppose $t\neq 0$. In this case, by making the change of variables $u_i\mapsto tu_i$, $i=1,2,3$ we see that $W'(t)$ is isomorphic to the variety $W''(-{\overline{ab^2}}ht^2)$, where $W''(t)$ is given by
$$
\begin{cases}
(b_1-b_4)u_1+(b_2-b_4)u_2-(b_3-b_4)u_3=1\\ \overline{u_1}^2+\overline{u_2}^2-\overline{u_3}^2-\overline{(u_1+u_2+u_3)}^2=t.
\end{cases}
$$
Let
$$
\alpha=\frac{b_1-b_4}{b_3-b_4},\,\beta=\frac{b_2-b_4}{b_3-b_4},\,\lambda=\frac{1}{b_3-b_4}.
$$
Then by forgetting variable $u_3$, we see that $W''(t)$ is isomorphic to $W_{\lambda}(t)$, where $W_{\lambda}(t)$ is the variety considered in the previous case and given by \eqref{W(t)}. But we already proved that, for every $\bm{b}\in \mathcal{B}\backslash\mathcal{B}^{\Delta}$, $W_{\lambda}(t)$ is an irreducible curve over $\overline{\F_q}$ for all but finitely many $t$. Thus, the inequality
\begin{equation}\label{desired}
\Sigma_2(K_2,\bm{b},h)\ll q
\end{equation}
also follows from Lemma \ref{Hoo} in this case. At least when $h\neq 0$.

Finally, if $h=0$, our goal is to modify the sum $\Sigma_2(K_2,\bm{b},h)$ by a change of variables and recover a case that was already considered before. We start by fixing $\xi$ any non-quadratic residue modulo $q$. Notice that for every $x\in \Z/q\Z$ there exists exactly two solutions to the equation
$$
x=\eta y^2,
$$
with $\eta\in\{1,\xi\}$ and $y\in \Z/q\Z$. With that in mind, we see that
$$
\Sigma_2(K_2;\bm{b},0)=\frac12\sum_{\eta\in\{1,\xi\}}\underset{r,s\rmod q}{\sum\sum}\prod_{i=1}^2K_2(\eta s^2(r+b_i)^2)\overline{K_2(\eta s^2(r+b_{i+2})^2)}.
$$
We see from definition \eqref{2ndK}, that whenever $t\neq 0$, we have the identity
$$
K_2(\eta t^2)=K_{1,\eta}(t),
$$
where $K_{1,\eta}$ is given by the right-hand-side of \eqref{1stK} with $a$ replaced by $\eta a$.
By treating the cases where $s(r+b_i)=0$ separately, we have that
\begin{align*}
\Sigma_2(K_2;\bm{b},0)&=\frac12\sum_{\eta\in\{1,\xi\}}\underset{r,s\rmod q}{\sum\sum}\prod_{i=1}^2K_{1,\eta}(s(r+b_i))\overline{K_{1,\eta}(s(r+b_{i+2}))}+O(q)\\
&=\frac12\sum_{\eta\in\{1,\xi\}}\Sigma_1(K_{1,\eta},\bm{b},0)+O(q).
\end{align*}
Therefore, \eqref{desired} for $h = 0$ follows from the first case considered above. This concludes the proof of Proposition \ref{completesum} provided that we can prove that we can impose $|\mathcal{B}^{\Delta}|\ll B^2$.

\subsection{The choice of $\mathcal{B}^{\Delta}$ and proof of Theorems \ref{expsum1} and \ref{expsum2}}

Let 
$$
\mathcal{D}=\left\{\bm{b}\in\mathcal{B};\; (b_1,b_3)=(b_2,b_4)\right\}.
$$
Let $\mathcal{E}\in \overline{\F_q}^2$ be the finite set given by Lemma \ref{UV-sq}. Then for each $(\alpha,\beta)\in \mathcal{E}\cap \F_q^2$, let $\mathcal{S}_{\alpha,\beta}$ be the set of solutions $\bm{b}\in\mathcal{B}$ of the linear system \eqref{system} and $\mathcal{S}_{\alpha,\beta}^{\ast}$ be the set of solutions to the dual system, obtained by replacing the roles of $(b_1,b_2)$ and $(b_3,b_4)$. Notice that
$$
\left|\mathcal{D}\right|=\left|\mathcal{S}_{\alpha,\beta}\right|=\left|\mathcal{S}_{\alpha,\beta}^{\ast}\right|=B^2.
$$
Finally, we put
$$
\mathcal{B}^{\Delta} = \mathcal{D}\cup \displaystyle\bigcup_{(\alpha,\beta)\in \mathcal{E}\cap \F_q^2} \left(\mathcal{S}_{\alpha,\beta} \cup \mathcal{S}_{\alpha,\beta}^{\ast}\right).
$$
Notice that this choice clearly satisfies the inequality 
$$
\mathcal{B}^{\Delta}\leq  30B^2,
$$
As we saw this was the last missing part in the proof of Proposition \ref{completesum}.

We must now put together the bounds for $\Sigma_j(K_j,\bm{b})$ in the cases where $\bm{b}\in \mathcal{B}^{\Delta}$ and $\bm{b}\in\mathcal{B}\backslash \mathcal{B}^{\Delta}$. Combining \eqref{diagonal}, \eqref{completion} and Proposition \ref{completesum}, we obtain
$$
\sum_{\bm{b}\in\mathcal{B}}\eta(b)\Sigma(K_j,\bm{b})\ll A^jB^2Mq+ B^4q\log q.
$$
The inequality \eqref{Holder} now gives
\begin{equation}\label{almostthere}
AB\times S_{K_j,j}\ll q^{\epsilon}(AN)^{3/4}(\|\bm{\alpha}\|_1\|\bm{\alpha}\|_2)^{1/2}\left( A^jB^2Mq+ B^4q \right)^{1/4}.
\end{equation}
We make the choices
\begin{equation}\label{choices}
A=N^{\frac{2}{j+2}}M^{-\frac{1}{j+2}},\,\,B=N^{\frac{j}{j+2}}M^{\frac{1}{j+2}},
\end{equation}
so that the conditions \eqref{necessary} become equivalent to
$$
M\leq N^2\text{ and }MN^{j}\leq q^{\frac{j+2}{2}},
$$
which are part of the hypotheses in Theorems \ref{expsum1} and \ref{expsum2}. With the choices as in \eqref{choices}, Inequality \eqref{almostthere} becomes
$$
S_{K_j,j}\ll q^{\epsilon}(\|\bm{\alpha}\|_1\|\bm{\alpha}\|_2)^{1/2}M^{1/4}N\left(\frac{q^{j+2}}{M^{j+1}N^{j+4}}\right)^{\frac{1}{4(j+2)}}\,\,(j=1,2),
$$
which proves both Theorem \ref{expsum1} and Theorem \ref{expsum2}.
\section{Proof of Theorem \ref{2/3}}\label{proofof23}

Let $q$ be a prime number, let $a$ be coprime with $q$ and $X\geq q$. We consider $E=E(X,q,a)$ given by

$$
E:=\sum_{\substack{n\leq X\\n\equiv a \rmod q}}\mu^2(n) - \frac{1}{\varphi(q)}\sum_{\substack{n\leq X\\(n,q)=1}}\mu^2(n).
$$
Our goal is to prove that for every $A>0$, we have the inequality $E\ll X/q(\log X)^A$
uniformly for $q\leq X^{13/19-\epsilon}$, where the iéplied constant depends at most on $\epsilon$ and $A$.

We use the classical identity
\begin{equation}\label{mu-decomp}
\mu^2(n)=\underset{\substack{n_1,n_2\geq 1\\n_1n_2^2=n}}{\sum\sum}\mu(n_2),
\end{equation}
giving
$$
E=\sum_{n\leq X^{1/2}}\mu(n)\Delta(X/n^2,q,a{\bar n}^2),
$$
where for every $x\geq 1$, $q$ integer and $a\in \Z/q\Z$,
$$
\Delta(x,q,a):=\sum_{\substack{m\leq x\\m\equiv a \rmod q}}1 - \frac{1}{\varphi(q)}\sum_{\substack{m\leq x\\(m,q)=1}}1.
$$
It is clear that for any $x,q,a$, we have
$$
\Delta(x,q,a) \ll 1.
$$
Let $N_0$ be a parameter to be chosen optimally later such that $1\leq N_0\leq X^{1/2}$. The previous inequality shows us that
\begin{equation}\label{N0-out}
E=\sum_{N_0< n\leq X^{1/2}}\mu(n)\Delta(X/n^2,q,a{\bar n}^2) + O(N_0).
\end{equation}
Notice that
\begin{align*}
\frac{1}{\varphi(q)}\sum_{N_0<n\leq X^{1/2}}\mu(n)
\sum_{m\leq X/n^2}1\ll\frac{X}{N_0q},
\end{align*}
since $q$ is a prime number. This and \eqref{N0-out} combined give
\begin{equation}\label{beforedyadic}
|E|\leq \sum_{N_0< n\leq X^{1/2}}\sum_{\substack{m\leq X/n^2\\m\equiv a{\bar n}^2\rmod q}}1 + O\left(N_0+\frac{X}{N_0 q}\right).    
\end{equation}

We now proceed by means of a dyadic decomposition. Let $V$ be a infinitely differentiable function defined on the real line vanishing outside $[1/2,4]$ and identical to $1$ in $[1,2]$. If we put

\begin{equation}\label{SV}
S_V(M,N;q,a)=\underset{\substack{m,n\\ mn^2\equiv a\rmod{q}}}{\sum\sum}V\left(\frac mM\right)V\left(\frac nN\right),
\end{equation}
we deduce from \eqref{beforedyadic} the upper bound
$$
E\ll (\log X)^2\cdot \sup_{M,N}S_V\left(M,N;q,a\right) + N_0 + \frac{X}{N_0 q},
$$
where the supremum is taken over all $M$ and $N$ such that
\begin{equation}\label{cond1}
M,N\geq 1,\, N_0\leq N\leq 2X^{1/2},\,MN^2\leq 8X.
\end{equation}

Let $M_0\geq 1$ be a parameter to be chosen optimally later. Suppose that $M\leq M_0$ and that $M,N$ satisfy the conditions \eqref{cond1}. Then, by the crude estimate 
$$
\sum_{n \equiv \alpha\rmod{q}}V(\frac{n}{N})\ll \left(\frac{N}{q} + 1\right),
$$
we see that
\begin{align*}
S_V(M,N;q,a)\ll M\left(\frac{N}{q}+1\right)\\
\ll \frac{X}{N_0q} + M_0.
\end{align*}

Suppose now that $MN^2\leq q^{\frac{101}{100}}$. In this case, we write $u=mn^2$ so that we obtain the inequality
$$
S_V(M,N;q,a)\ll \sum_{\substack{u\leq 8q^{\frac{101}{100}}\\ u\equiv a\rmod{q}}}d(u)\ll q^{\frac{1}{100} + \epsilon},
$$
where we used the classical bound $d(n)\ll n^{\epsilon}$ for every $\epsilon>0$. Putting everything together we see that
\begin{equation}\label{E-sup}
E\ll (\log X)^2\sup_{M,N}S_V\left(M,N;q,a\right) +q^{\frac1{100}+\epsilon} + M_0 + N_0 + \frac{X}{N_0 q},
\end{equation}
where now the supremum is taken over all $M$ and $N$ satisfying
\begin{equation}\label{cond2}
M\geq M_0,\, N\geq N_0,\,q^{\frac{101}{100}}\leq MN^2\leq 8X.
\end{equation}

In the next subsection, we will use Theorem \ref{expsum1} to estimate $S_V(M,N;q,a)$, but before doing that, we need some preparation. Indeed, we use Poisson summation in both variables and than we separate the contribution coming from the main terms.

\subsection{Double Poisson summation}

Let $q$ be a prime number and $a$ be coprime with $q$. Let $M$ and $N$ be real numbers satisfying \eqref{cond2}. Let $S_V(M,N;q,a)$ be given by \eqref{SV}, then by applying Poisson summation in both variables, we get
$$
S_V(M,N;q,a)=\frac{MN}{q^2}\underset{m,n}{\sum\sum}\widehat{V}\left(\frac{mM}{q}\right)\widehat{V}\left(\frac{nN}{q}\right)\sideset{}{^{\ast}}\sum_{u\rmod q}e_q\left(m{\bar u}^2+anu\right).
$$
We first notice that since $V$ is smooth, integrating by parts gives the inequalities
\begin{equation}\label{RapDecay}
\widehat{V}(x)\ll x^{-j}, \;x\in \R,\, j=0,1,2,\ldots
\end{equation}
Hence, it follows that for every $\epsilon>0$, the contribution of the terms where $|m|>q^{1+\epsilon}M^{-1}$ or $|n|>q^{1+\epsilon}N^{-1}$ is negligible. For instance, we have
\begin{equation}\label{useRapDecay}
S_V(M,N;q,a)=\frac{MN}{q^2}\underset{\substack{|m|\leq q^{1+\epsilon}M^{-1}\\ |n|\leq q^{1+\epsilon}N^{-1} }}{\sum\sum}\widehat{V}\left(\frac{mM}{q}\right)\widehat{V}\left(\frac{nN}{q}\right)S(m,an;q) + O(q^{-200}),
\end{equation}
where $S(m,an;q)$ is as defined in \eqref{Smnq}.

The contribution of the terms where $mn=0$ can also be estimated easily by directly computing the exponential sums and using the estimates \eqref{RapDecay} with $j=0$. Indeed, if $0<|m|,|n|<q$, we have the following identities:

\begin{equation}\label{calc}
\begin{cases}
|S(m,0;q)|=q^{1/2},\\
S(0,an;q)=-1,\\
S(0,0;q)=q-1,
\end{cases}
\end{equation}
since the first of this sums is a Gauss sum, the second one is a Ramanujan sum and the last one is a trivial sum.

Suppose $\epsilon$ satisfies $M_0, N_0\geq q^{\epsilon}$. We see from \eqref{calc} that one has the upper bound
$$
\underset{\substack{|m|\leq q^{1+\epsilon}M^{-1},\, |n|\leq q^{1+\epsilon}N^{-1} \\ mn=0}}{\sum\sum}\widehat{V}\left(\frac{mM}{q}\right)\widehat{V}\left(\frac{nN}{q}\right)S(m,an;q) \ll   q^{3/2+\epsilon}M^{-1} + q^{1+\epsilon}N^{-1} + q.
$$
By \eqref{useRapDecay}, we see that
\begin{equation}\label{SV=T}
S_V(M,N;q,a)=\frac{MN}{q^2}\mathcal{T} + O\left(q^{\epsilon}\left(\frac{N}{q^{1/2}} + \frac{MN}{q}\right)\right),
\end{equation}
where
\begin{align}\label{T=Teps}
\mathcal{T}&=\sum_{\eps_1=\pm 1}\sum_{\eps_2=\pm 1}\sum_{m=1}^{q^{1+\epsilon}M^{-1}}\sum_{n=1}^{q^{1+\epsilon}N^{-1}}\widehat{V}\left(\frac{\eps_1mM}{q}\right)\widehat{V}\left(\frac{\eps_2nN}{q}\right)S(\eps_1m,\eps_2an;q)\\
&=:\sum_{\eps_1=\pm 1}\sum_{\eps_2=\pm 1}\mathcal{T}_{\eps_1,\eps_2},\notag
\end{align}
say. We must now estimate $\mathcal{T}_{\eps_1,\eps_2}$. By integration by parts and the trivial upper bounds
$$
\widehat{V}(x),\,\widehat{V}'(x)\ll 1,
$$
we deduce the inequality
\begin{equation}\label{Teps}
\mathcal{T}_{\eps_1,\eps_2}\ll q^{2\epsilon}\sup_{M^{\ast},N^{\ast}}\sum_{m=1}^{M^{\ast}}\sum_{n=1}^{N^{\ast}}S(\eps_1m,\eps_2an;q),
\end{equation}
where the supremum is taken over all $M^{\ast}$ and $N^{\ast}$ such that
\begin{equation}\label{cond*}
1\leq M^{\ast}\leq q^{1+\epsilon}M^{-1},\,\, 1\leq N^{\ast}\leq q^{1+\epsilon}N^{-1}.
\end{equation}
We are now ready to use Theorem \ref{expsum1}. We prove the following

\begin{prop}\label{majoration}
Let $q$ be a prime number. Let $a$ and $b$ be coprime with $q$. Let $M, N\geq 1$ be such that
$$
M,N<q,\,MN^2< q^2.
$$
Let $S(m,n;q)$ be as in \eqref{Smnq}. Then for any $\epsilon>0$, we have
$$
\sum_{m\leq M}\sum_{n\leq N}S(am,bn;q)\ll MNq^{1/2+\epsilon}\left(\frac{M^3N^6}{q^4}\right)^{-1/16} + M^{3/2}q^{1/2+\epsilon},
$$
where the implied constant only depends on $\epsilon$.
\end{prop}

\begin{proof}
There are two cases to consider. First, if $M\leq N^2$, the proposition follows from Theorem \ref{expsum1} with $\alpha_m=1$ for every $1\leq m\leq M$, $\mathcal{N}=[1,N]$ and
$$
K(t)=\frac{1}{q^{1/2}}S(at,b;q).
$$
Indeed, for $1\leq n<q$, we have
$$
S(am,bn;q)=S(amn^2,b;q).
$$
On the other hand, if $M>N^2$, a simple application of the Weil bound \eqref{Weil} gives
$$
\sum_{m\leq M}\sum_{n\leq N}S(m,n;q)\ll MNq^{1/2} \leq  M^{3/2}q^{1/2}.
$$
This concludes the proof of the proposition.
\end{proof}
We want to apply this proposition to the right-hand side of \eqref{Teps}. In order to do so, we need to be sure that any $M^{\ast},\, N^{\ast}$ satisfying \eqref{cond*} will also satisfy the conditions of Proposition \ref{majoration}. It suffices to have
$$
M,N>q^{\epsilon}, MN^2> q^{1+3\epsilon}.
$$
By \eqref{cond2}, this follows from the assumptions
\begin{equation}\label{lastcond}
M_0,N_0>q^{\epsilon}\text{ and }\epsilon<\frac{1}{300}.
\end{equation}
Assume \eqref{lastcond}. Then Proposition \ref{majoration} applied to the right-hand side of \eqref{Teps} gives
$$
\mathcal{T}_{\eps_1,\eps_2}\ll \frac{q^{5/2+\epsilon}}{MN}\left(\frac{M^3N^6}{q^5}\right)^{1/16} + \frac{q^{2}}{M^{3/2}}.
$$
This together with \eqref{SV=T} and \eqref{T=Teps} gives
$$
S_V(M,N;q,a)\ll q^{\epsilon}\left((MN^2)^{3/16}q^{3/16} + NM^{-1/2} + Nq^{-1/2} + MNq^{-1}\right).
$$
We now see from \eqref{E-sup} and \eqref{cond2} that we have the inequality
$$
E\ll q^{\epsilon}\left(X^{3/16}q^{3/16} + \frac{X^{1/2}}{M_0} + \frac{X^{1/2}}{M_0^{1/2}q^{1/2}} + \frac{X}{N_0q} + q^{\frac1{100}} + M_0 + N_0\right).
$$
We make the choices (clearly satisfying \eqref{lastcond})
$$
M_0=X^{1/4}\text{ and }N_0=(X/q)^{1/2},
$$
thus obtaining
$$
E\ll q^{\epsilon}\left(X^{3/16}q^{3/16} + X^{1/4} + X^{1/2}q^{-1/2} \right).
$$
It is now easy to see that for every $\epsilon,A>0$ and whenever $q\leq X^{13/19-\epsilon}$, then 
$$
E\ll \frac{X}{q(\log X)^A}.
$$
We are now done proving Theorem \ref{2/3}.

\section*{Acknowledgements}
It is a pleasure to thank \'Etienne Fouvry and Philippe Michel for very useful discussions on the subject of this article.

\bibliographystyle{amsplain}
\bibliography{references}

\end{document}